\documentclass[a4paper,11pt]{article}

\usepackage{amsmath,amsthm,amssymb,mathrsfs,latexsym,amsfonts}
\usepackage{graphicx,psfrag,epsfig}
\usepackage{bbm}
\usepackage[english]{babel}
\usepackage[latin1]{inputenc}
\usepackage{a4wide}

\newtheorem{theorem}{Theorem}
\newtheorem{lemma}[theorem]{Lemma}

\newcounter{paraga}[subsection]

\begin{document}

\def\MP{\,{<\hspace{-.5em}\cdot}\,}
\def\SP{\,{>\hspace{-.3em}\cdot}\,}
\def\PM{\,{\cdot\hspace{-.3em}<}\,}
\def\PS{\,{\cdot\hspace{-.3em}>}\,}
\def\EP{\,{=\hspace{-.2em}\cdot}\,}
\def\PP{\,{+\hspace{-.1em}\cdot}\,}
\def\PE{\,{\cdot\hspace{-.2em}=}\,}
\def\N{\mathbb N}
\def\C{\mathbb C}
\def\Q{\mathbb Q}
\def\R{\mathbb R}
\def\T{\mathbb T}
\def\A{\mathbb A}
\def\Z{\mathbb Z}
\def\demi{\frac{1}{2}}

\begin{titlepage}
\author{Abed Bounemoura~\footnote{CNRS - CEREMADE, Universit\'e Paris Dauphine \& IMCCE, Observatoire de Paris (abedbou@gmail.com)} {} and Vadim Kaloshin~\footnote{University of Maryland at College Park (vadim.kaloshin@gmail.com)}}
\title{\LARGE{\textbf{A note on micro-instability for Hamiltonian systems close to integrable}}}
\end{titlepage}

\maketitle

\begin{abstract}
In this note, we consider the dynamics associated to a perturbation of an integrable Hamiltonian system in action-angle coordinates in any number of degrees of freedom and we prove the following result of ``micro-diffusion": under generic assumptions on $h$ and $f$, there exists an orbit of the system for which the drift of its action variables is at least of order $\sqrt{\varepsilon}$, after a time of order $\sqrt{\varepsilon}^{-1}$. The assumptions, which are essentially minimal, are that there exists a resonant point for $h$ and that the corresponding averaged perturbation is non-constant. The conclusions, although very weak when compared to usual instability phenomena, are also essentially optimal within this setting. 
\end{abstract}

\section{Introduction and result}\label{s1}

\subsection{Introduction}\label{s11}

Let $n \geq 2$ be an integer, $B=B_1 \subseteq \R^n$ be the unit open ball with respect to the supremum norm $|\,.\,|$ and $\T^n:=\R^n / \Z^n$. Consider a smooth (at least $C^2$) Hamiltonian function $H$ defined on the domain $\T^n \times B$ of the form
\begin{equation}\label{Ham}\tag{H}
H(\theta,I)=h(I)+\varepsilon f(\theta,I), \quad \varepsilon \geq 0, \quad (\theta,I) \in \T^n \times B, 
\end{equation} 
and its associated Hamiltonian system
\begin{equation*}
\begin{cases} 
\dot{\theta}(t)=\partial_I H(\theta(t),I(t))=\partial_I h(I(t))+ \varepsilon \partial_I f(\theta(t),I(t)), \\
\dot{I}(t)=- \partial_\theta H(\theta(t),I(t))=- \varepsilon \partial_\theta f(\theta(t),I(t)). 
\end{cases}
\end{equation*} 
For $\varepsilon=0$, the system is stable in the sense that the action variables $I(t)$ of all solutions are constant, and these solutions are quasi-periodic. Now for $\varepsilon \neq 0$ but sufficiently small, the celebrated KAM theorem (\cite{Kol54}, \cite{Arn63a}, \cite{Mos62}) and Nekhoroshev theorem (\cite{Nek77}, \cite{Nek79}) assert that the system, provided it is real-analytic, retains some stability properties: for a generic $h$ and all $f$, ``most" solutions are quasi-periodic and the action variables of all solutions are almost constant for a very long interval of time.

Yet in the same setting, Arnold conjectured in the sixties that for a generic $h$ and for $n \geq 3$, the following phenomenon of instability should occur: ``for any points $I'$ and $I''$ on the connected level hypersurface of $h$ in the action space there exist orbits connecting an arbitrary small neighborhood of the torus $I=I'$ with an arbitrary small neighborhood of the torus $I=I''$, provided that $\varepsilon$ is sufficiently small and that $f$ is generic" (see \cite{Arn94}).

Since Arnold's original example of such a phenomenon (\cite{Arn64}), this question has been investigated extensively, but only recently solutions to this conjecture have appeared for convex $h$  (see \cite{KZ12}, \cite{Che13}
for $n=3$ and \cite{KZ14a} for $n=4$ and a progress for any $n>4$ in \cite{BKZ}, \cite{KZ14b}). For non-convex integrable Hamiltonians that posses a ``super-conductivity channel" (that is, a rational subspace contained in an energy level), it is very simple to construct examples of perturbation having unstable solutions (see \cite{Mos60}, \cite{Nek79}), and this is also true for a generic perturbation for $n=2$ (\cite{BK14}). Apart from these two classes of integrable Hamiltonians (the convex ones and the ones that posses a super-conductivity channel), nothing is known, even for a specific perturbation. 

It is our purpose here to show, using the method of \cite{BK14}, that for a generic integrable Hamiltonian (this generic condition being the existence of a resonant point) and for a generic perturbation (the associated averaged perturbation is non-constant), one has a phenomenon of ``micro-instability": existence of a solution whose action variables drift of order $\sqrt{\varepsilon}$ after a time of order $\sqrt{\varepsilon}^{-1}$.

\subsection{Result}\label{s12}

Let $H$ be as in~\eqref{Ham}, we assume it is of class $C^3$ and
\begin{equation}\label{bound}
|h|_{C^2(B)} \leq 1, \quad |f|_{C^3(\T^n \times B)} \leq 1
\end{equation}
where $|\,.\,|_{C^2(B)}$ (respectively $|\,.\,|_{C^3(\T^n \times B)}$) denotes the standard $C^2$-norm for functions defined on $B$ (respectively the standard $C^3$-norm for functions defined on $\T^n \times B$). Our first general assumption is on the integrable Hamiltonian $h$.

\bigskip

$(A.1)$. There exists $I^* \in B$ such that $\omega:=\partial_I h(I^*)$ is resonant but non-zero, that is $k \cdot \omega=0$ for some (but not all) $k \in \Z^n \setminus\{0\}$, where $\cdot$ denotes the Euclidean scalar product. 

\bigskip

We denote by $\Lambda$ the real subspace of $\R^n$ spanned by the $\Z$-module $\{k \in \Z^n \; | \; k \cdot \omega=0\}$. By assumption, the dimension of $\Lambda$ is at least $1$ and at most $n-1$. Without loss of generality, we will assume that $I^*=0$. For our second assumption, we define
\[ f_{\omega}(\theta,I):=\lim_{t \rightarrow +\infty}\frac{1}{t}\int_0^t f(\theta+s\omega,I)ds, \quad f_{\omega}^*(\theta):=f_\omega(\theta,I^*)=f_\omega(\theta,0). \]
Expanding $f$ in Fourier series, $f(\theta,I)=\sum_{k \in \Z^n}f_k(I)e^{i2\pi k\cdot \theta}$, we have a more explicit expression
\[ f_{\omega}^*(\theta)=\sum_{k \in \Z^n \cap \Lambda}f_k(I^*)e^{i2\pi k\cdot \theta}=\sum_{k \in \Z^n \cap \Lambda}f_k(0)e^{i2\pi k\cdot \theta}. \] 
Our second general assumption is as follows.

\bigskip

$(A.2)$ The function $f_{\omega}^*$ is non-constant, that is there exists $\theta^* \in \T^n$ such that $|\partial_\theta f_\omega^*(\theta^*)|=\lambda>0$. 

\bigskip

From~\eqref{bound} we necessarily have $\lambda\leq 1$. In order to state precisely our theorem, we need further definitions. Let $\Lambda^\perp$ be the orthogonal complement of $\Lambda$ (observe that $\Lambda^\perp$ is nothing but the minimal rational subspace of $\R^n$ containing $\omega$). Let us define $\Psi=\Psi_{\omega}$ by
\begin{equation}\label{funcpsi}
\Psi(Q)=\max\left\{|k\cdot\omega|^{-1} \; | \; k\in \Lambda^\perp \cap \Z^n, \; 0<|k|\leq Q \right\}.
\end{equation} 
This is well-defined for $Q\geq Q_\omega$, where $Q_\omega \geq 1$ is a constant depending on $\omega$ (see \cite{BF13}). Then for $x \geq Q_\omega\Psi(Q_\omega)$, we define $\Delta=\Delta_{\omega}$ by
\begin{equation}\label{funcdelta}
\Delta(x)=\sup\{Q \geq Q_\omega\; | \; Q\Psi(Q)\leq x\}.
\end{equation} 
We can finally state our theorem.

\begin{theorem}\label{th}
Let $H$ be as in~\eqref{Ham} satisfying~\eqref{bound}, and assume that $(A.1)$ and $(A.2)$ holds true. There exist positive constants $\kappa=\kappa(n,|\omega|,\Lambda)$, $\mu_0=\mu_0(n,|\omega|,\Lambda,\lambda)$, $c=c(\lambda,\Lambda)$ and $\delta=\delta(\lambda,\Lambda)$ such that if
\[ 0<\mu(\sqrt{\varepsilon}):=\left(\Delta\left(\kappa\sqrt{\varepsilon}^{-1}\right)\right)^{-1}\leq \mu_0, \]
then there exists a solution $(\theta(t),I(t))$ of the system~\eqref{Ham} such that
\[ |I(\tau)-I(0)| \geq c \sqrt{\varepsilon}, \quad \tau:=\delta/\sqrt{\varepsilon}.  \]
Moreover, there exists a positive constant $C=C(n,|\omega|,\Lambda)$ such that this solution satisfies
\[ d(I(0),I^*)\leq C\sqrt{\varepsilon}\mu(\sqrt{\varepsilon}), \quad d(I(t)-I(0),\Lambda) \leq C\sqrt{\varepsilon}\mu(\sqrt{\varepsilon}), \quad 0\leq t\leq \tau \]
where $d$ is the distance induced by the supremum norm.
\end{theorem}

Observe that $\mu(\sqrt{\varepsilon})$ always converge to zero as $\varepsilon$ goes to zero, more slowly than $\sqrt{\varepsilon}$: for instance, if $\omega$ is periodic (a multiple of a rational vector), then $\mu(\sqrt{\varepsilon})$ is exactly of order $\sqrt{\varepsilon}$, and if $\omega$ is resonant-Diophantine (meaning that it is not rational but the function $\Psi$ defined above grows at most as a power), then $\mu(\sqrt{\varepsilon})$ is of order a power of $\sqrt{\varepsilon}$. In general, the speed of convergence to zero can be arbitrarily slow. Yet the quantity $\sqrt{\varepsilon}\mu(\sqrt{\varepsilon})$ is always smaller than $\sqrt{\varepsilon}$, and so the statement implies that the $\sqrt{\varepsilon}$-drift occurs along the resonant direction $\Lambda$, as in the transverse direction the variation of the action is of order $\sqrt{\varepsilon}\mu(\sqrt{\varepsilon})$ during the interval of time considered. 

\subsection{Some comments}\label{s13}

Let us now briefly discuss the assumptions and conclusions of Theorem~\ref{th}. 

First, if the assumption $(A.1)$ is not satisfied, that is if the image of the gradient map $\partial_I h$ does not contain a resonant point (which means that this image is contained in a non-resonant line), it is not hard to see that the conclusions of the theorem do not hold true: for all solutions and for all $0 \leq t \leq \tau$, the variation of the action variables cannot be of order $\sqrt{\varepsilon}$ (or put it differently, in order to have a drift of order $\sqrt{\varepsilon}$ one needs a time strictly larger than $\tau$). Indeed, one can prove in this case (using normal form techniques) that the system can be (globally) conjugated to another system which consists of an integrable part plus a perturbation whose size is of order $\varepsilon\mu(\sqrt{\varepsilon})$: this implies that for times $0 \leq t \leq \tau$, the variation of the action variables of all solutions is of order at most $\sqrt{\varepsilon}\mu(\sqrt{\varepsilon})$. Then, if the assumption $(A.2)$ is not satisfied, it is also easy to see that the conclusions of the theorem do not hold true for any solution starting close to $I^*=0$: indeed, looking at Lemma~\ref{normal} below, one would get a (local, defined around $I^*=0$) conjugacy to a perturbation of an integrable system, with a perturbation whose size is again of order $\varepsilon\mu(\sqrt{\varepsilon})$. 

Concerning the conclusions, it is also plain to remark that at the time $\tau$ the variation of the action variables cannot be larger than $\sqrt{\varepsilon}$, up to a constant. But more is true in the special case where $h$ is convex (or quasi-convex) and $\Lambda$ is a hyperplane (which is equivalent to $\omega$ being a periodic vector): for a time $T$ which is very large (any fixed power of $\sqrt{\varepsilon}^{-1}$ if $H$ is smooth of even $\exp\left(\sqrt{\varepsilon}^{-1}\right)$ if $H$ is real-analytic), the variation of the action variables of the solution given by Theorem~\ref{th} is of order $\sqrt{\varepsilon}$ for times $0 \leq t \leq T$ (see \cite{Loc92} for the analytic case and \cite{Bou10} for the smooth case). In this situation, one has the curious fact that the variation of the action variables is exactly of order $\sqrt{\varepsilon}$ (in the sense that it is bigger than some small constant times $\sqrt{\varepsilon}$ and smaller than some large constant times $\sqrt{\varepsilon}$) during the very long interval of time $\tau \leq t \leq T$.

\section{Proof of the result}\label{s2}

The proof of Theorem~\ref{th} follows the strategy of \cite{BK14}. On a $\sqrt{\varepsilon}$-neighborhood of the point $I^*$, we will conjugate our Hamiltonian to a simpler Hamiltonian (a resonant normal form plus a small remainder) for which the result will be obtained by simply looking at the equations of motion. Using the fact that the conjugacy is given by a symplectic transformation which is close to identity, the result for our original Hamiltonian will follow. The normal form will be stated in \S\ref{s21}, and the proof of Theorem~\ref{th} will be given in \S\ref{s22}. 

\subsection{A normal form lemma}\label{s21}

Before starting the proof, it will be more convenient to assume that the subspace $\Lambda$ is generated by the first $d$ vectors of the canonical basis of $\R^n$, for $1 \leq d \leq n-1$. This is no restriction, as by a linear symplectic change of coordinates one can always write $\omega=(0,\tilde{\omega})\in \R^d \times \R^{n-d}$ for some non-resonant vector $\tilde{\omega} \in \R^{n-d}$. This enables us to get rid of the dependence on $\Lambda$ in the constants involved. Since $\Lambda^\perp \cap \Z^n=\{0\} \times \Z^{n-d}$, the function $\Psi$ defined in~\eqref{funcpsi} takes the simpler form
\begin{equation*}
\Psi(Q)=\max\left\{|k\cdot\tilde{\omega}|^{-1} \; | \; k \in \Z^{n-d}, \; 0<|k|\leq Q \right\}
\end{equation*} 
and is well-defined for $Q\geq 1$ (that is one can take $Q_\omega=1$). The function $\Delta$ introduced in~\eqref{funcpsi} is then defined for $x \geq \Psi(1)=|\tilde{\omega}|^{-1}$ and we have
\begin{equation*}
\Delta(x)=\sup\{Q \geq 1\; | \; Q\Psi(Q)\leq x\}.
\end{equation*} 
Observe also that in this situation, we simply have
\[ f_\omega(\theta,I)=f_\omega(\theta_1,\dots,\theta_d,I)=\int_{\T^{n-d}}f(\theta_1,\dots,\theta_d,\theta_{d+1},\dots,\theta_n,I)d\theta_{d+1}\dots d\theta_{n}\]
and hence
\[ f_\omega^*(\theta_1,\dots,\theta_d)=f_\omega(\theta_1,\dots,\theta_d,I^*)=f_\omega(\theta_1,\dots,\theta_d,0). \]
The assumption $(A.2)$ thus reduces to the existence of a point $\theta^*=(\theta_1^*,\dots,\theta_d^*) \in \T^d$ and a constant $0<\lambda\leq 1$ such that $|\partial_\theta f_\omega^*(\theta^*)|=\lambda$. Here's the statement of our normal form lemma.

\begin{lemma}\label{normal}
Let $H$ be as in~\eqref{Ham} satisfying~\eqref{bound}. There exist positive constants $\kappa=\kappa(n,|\omega|)$, $\mu_0=\mu_0(n,|\omega|)$ and $C=C(n,|\omega|)$ such that if 
\begin{equation}\label{seuil}
\mu(\sqrt{\varepsilon}):=\Delta\left(\kappa\sqrt{\varepsilon}^{-1}\right)^{-1}\leq \mu_0,
\end{equation}
then there exists a symplectic map $\Phi: \T^n \times B_{2\sqrt{\varepsilon}} \rightarrow \T^n \times B_{3\sqrt{\varepsilon}}$ of class $C^2$ such that
\[ H \circ \Phi(\theta,I)= h(I)+\varepsilon f_\omega(\theta_1,\dots,\theta_d,I)+\tilde{f}_\varepsilon(\theta,I)  \]
with the estimates 
\begin{equation}\label{estimates1}
|\Pi_I \Phi-\mathrm{Id}|_{C^0(\T^n \times B_{2\sqrt{\varepsilon}})} \leq C\sqrt{\varepsilon}\mu(\sqrt{\varepsilon}),   
\end{equation}
\begin{equation}\label{estimates2}
|\partial_\theta \tilde{f}_\varepsilon|_{C^0(\T^n \times B_{2\sqrt{\varepsilon}})} \leq C \varepsilon\mu(\sqrt{\varepsilon}),
\end{equation}
\begin{equation}\label{estimates3}
|\partial_I \tilde{f}_\varepsilon|_{C^0(\T^n \times B_{2\sqrt{\varepsilon}})} \leq C \sqrt{\varepsilon}\mu(\sqrt{\varepsilon}). 
\end{equation}
\end{lemma}

This is a very special case of Theorem 1.3 of \cite{Bou13b}, to which we refer for a proof (strictly speaking, Theorem 1.3 of \cite{Bou13b} would require in our situation the integrable Hamiltonian $h$ to be of class $C^5$, but one can see from the proof that $C^2$ is in fact sufficient).

\subsection{Proof of Theorem~\ref{th}}\label{s22}

Theorem~\ref{th} will be easily deduced from Lemma~\ref{normal}.

\begin{proof}[Proof of Theorem~\ref{th}]
We start by choosing $\varepsilon>0$ sufficiently small so that~\eqref{seuil} holds true. Then we apply Lemma~\ref{normal}: there exists a symplectic map $\Phi: \T^n \times B_{2\sqrt{\varepsilon}} \rightarrow \T^n \times B_{3\sqrt{\varepsilon}}$ of class $C^2$ such that
\[ H \circ \Phi(\theta,I)= h(I)+\varepsilon f_\omega(\theta_1,\dots,\theta_d,I)+\tilde{f}_\varepsilon(\theta,I)  \]
with the estimates~\eqref{estimates1},~\eqref{estimates2} and~\eqref{estimates3}. Obviously $f_\omega$ has unit $C^3$-norm since this is the case for $f$. Consider the solution $(\theta(t),I(t))$ of the Hamiltonian $H \circ \Phi$, starting at $I(0)=I^*=0$, $(\theta_1(0),\dots,\theta_d(0))=\theta^* \in \T^d$ and with $(\theta_{d+1}(0),\dots,\theta_n(0)) \in \T^{n-d}$ arbitrary. It satisfies the following equations:
\begin{equation}\label{mouve}
\begin{cases}
\dot{I}(t)=-\varepsilon\partial_{\theta} f_\omega(\theta_1(t),\dots,\theta_d(t),I(t)) -\partial_{\theta} \tilde{f}_\varepsilon(\theta(t),I(t)), \\
\dot{\theta}(t)=\partial_I h(I(t))+\varepsilon\partial_{I} f_\omega(\theta_1(t),\dots,\theta_d(t),I(t))+\partial_{I} \tilde{f}_\varepsilon(\theta(t),I(t)).
\end{cases}
\end{equation}
Let us fix $\delta:=\sqrt{\lambda/6}$ and let $\tau=\delta/\sqrt{\varepsilon}$. From the first equation of~\eqref{mouve} and the estimate~\eqref{estimates2}, one has
\begin{equation}\label{est1}
|I(t)-I(0)| \leq \delta\sqrt{\varepsilon}+ \delta C \sqrt{\varepsilon}\mu(\sqrt{\varepsilon})\leq \delta\sqrt{\varepsilon}+\delta\sqrt{\varepsilon}= 2\delta\sqrt{\varepsilon}, \quad 0 \leq t \leq \tau,
\end{equation} 
up to taking $\mu_0$ smaller than $C^{-1}$. Using the fact that $\partial_I h(I(0))=\partial_I h(0)=(0,\tilde{\omega}) \in \R^d \times \R^{n-d}$ which follows from our first assumption and the choice of $I(0)$, this last estimate, together with the fact that $h$ has unit $C^2$-norm, imply that 
\[ \max_{1 \leq i \leq d}|\partial_{I_i} h(I(t))|=\max_{1 \leq i \leq d}|\partial_{I_i} h(I(t))-\partial_{I_i} h(I(0))| \leq 2\delta\sqrt{\varepsilon}, \quad 0 \leq t \leq \tau. \]
From the second equation of~\eqref{mouve} and the estimate~\eqref{estimates3}, we obtain from the last estimate
\[ \max_{1 \leq i \leq d}|\dot{\theta}_i(t)|\leq 2\delta\sqrt{\varepsilon}+\varepsilon+C \sqrt{\varepsilon}\mu(\sqrt{\varepsilon}), \quad 0 \leq t \leq \tau. \]
Taking $\mu_0$ small enough with respect to $C$ and $\lambda$ (and hence $\delta$), the sum of the last two terms of the right-hand side of the inequality above can me made smaller than $\delta\sqrt{\varepsilon}$, thus we can ensure that
\[ \max_{1 \leq i \leq d}|\dot{\theta}_i(t)|\leq 3\delta\sqrt{\varepsilon}, \quad 0 \leq t \leq \tau \]
and hence
\[ \max_{1 \leq i \leq d}|\theta_i(t)-\theta_i(0)|\leq 3\delta^2, \quad 0 \leq t \leq \tau. \]
Now from our second assumption and the choice of the initial condition, we have
\[ |\varepsilon\partial_{\theta} f_\omega(\theta_1(0),\dots,\theta_d(0),I(0))|=|\varepsilon\partial_{\theta} f_\omega^*(\theta_1^*,\dots,\theta_d^*)|=\varepsilon\lambda>0\] 
so, using the last estimate and our choice of $\delta$, we get 
\[ |\varepsilon\partial_{\theta} f_\omega(\theta_1(t),\dots,\theta_d(t),I(0))|\geq \varepsilon\lambda-3\varepsilon\delta^2=\varepsilon\lambda-\varepsilon\lambda/2=\varepsilon\lambda/2, \quad 0 \leq t \leq \tau.\]
From~\eqref{est1}, taking $\mu_0$ smaller with respect to $\lambda$, we can make sure that
\[ |\varepsilon\partial_{\theta} f_\omega(\theta_1(t),\dots,\theta_d(t),I(t))|\geq\varepsilon\lambda/3, \quad 0 \leq t \leq \tau\]
but also, from the estimate~\eqref{estimates2}, 
\[ |\varepsilon\partial_{\theta} f_\omega(\theta_1(t),\dots,\theta_d(t),I(t))+\partial_{\theta} \tilde{f}_\varepsilon(\theta(t),I(t))|\geq\varepsilon\lambda/4, \quad 0 \leq t \leq \tau.\]
From the first equation of~\eqref{mouve} we finally have
\[ |I(\tau)-I(0)|\geq \max_{1 \leq i \leq d}|I_i(\tau)-I_i(0)|\geq \sqrt{\varepsilon}\lambda\delta/4 =\sqrt{\varepsilon} 3\delta^3/2 \]
but also, using the estimate~\eqref{estimates2},
\[ \max_{d+1 \leq j \leq n}|I_j(t)-I_j(0)|\leq C\delta\sqrt{\varepsilon}\mu(\sqrt{\varepsilon})\leq C\sqrt{\varepsilon}\mu(\sqrt{\varepsilon}), \quad 0 \leq t \leq \tau.  \]
To conclude, using the estimate~\eqref{estimates1}, this solution for $H \circ \Phi$ gives rise to a solution for $H$ that, abusing notations, we still denote $(\theta(t),I(t))$ and such that, taking once again $\mu_0$ smaller with respect to $\lambda$, satisfies
\[ |I(\tau)-I(0)|\geq \sqrt{\varepsilon} \delta^3:=c\sqrt{\varepsilon} \]
and also, up to enlarging the constant $C$,
\[ d(I(0),I^*) \leq C\sqrt{\varepsilon}\mu(\sqrt{\varepsilon}), \quad d(I(t)-I(0),\R^d \times \{0\})\leq C\sqrt{\varepsilon}\mu(\sqrt{\varepsilon}).  \]
This concludes the proof.
\end{proof}

\bigskip

\noindent{\bf Acknowledgements.} The first author would like to thank IMPA for its hospitality. The second author acknowledges partial support of the NSF grant DMS-1402164.

\addcontentsline{toc}{section}{References}
\bibliographystyle{amsalpha}
\bibliography{Micro-diffusion}

\end{document}